\documentclass[11pt,a4paper]{article}

\usepackage{amsmath}
\usepackage{amsthm}
\usepackage{amssymb}
\usepackage{srcltx}
\usepackage{hyperref}
\usepackage{tikz}

\newtheorem{lemma}{Lemma}[section]
\newtheorem{theorem}[lemma]{Theorem}

\newtheorem{proposition}[lemma]{Proposition}

\theoremstyle{definition}

\newtheorem{remark}[lemma]{Remark}

\numberwithin{equation}{section}
\newtheorem{observation}[lemma]{Observation}

\DeclareMathOperator{\diag}{diag}

\newcommand{\leaveout}[1]{}

\makeatletter
\renewcommand{\p@enumii}{}
\makeatother

\newcommand\R{{\mathbb R}}
\newcommand\C{{\mathbb C}}
\newcommand\D{{\mathbb D}}

\newcommand\rplus{{\R_{+}}}

\newcommand{\set}[1]{\left\lbrace #1 \right\rbrace}

\newcommand{\bigmid}{\bigm\vert}

\newcommand{\bi}{\begin{itemize}}
\newcommand{\ei}{\end{itemize}}
\newcommand{\be}{\begin{enumerate}}
\newcommand{\ee}{\end{enumerate}}

\newcommand{\bbm}[1]{\begin{bmatrix}#1\end{bmatrix}}

\makeatletter

\def\etv{& \hskip-.3em\vrule\hskip-.3em &} 
\def\smalletv{&\vrule&} 
\def\smallcrh{\vrule height0pt depth2\ex@ width0pt
\cr\noalign{\hrule}
\vrule height6.5\ex@ depth0pt width0pt}
\newbox\smallstrutbox
\setbox\smallstrutbox=\hbox{\vrule height6pt depth1.5pt width\z@}
\def\smallstrut{\relax\ifmmode\copy\smallstrutbox\else\unhcopy\smallstrutbox\fi}

\newenvironment{sysmatrix}{
\let\|=\etv
\hskip \arraycolsep
\begin{matrix}}
{\end{matrix}
\hskip \arraycolsep
}       

\newenvironment{smallsysmatrix}{\null\,\vcenter\bgroup
\let\|=\smalletv

\def\\{\smallstrut\math@cr}
\restore@math@cr\default@tag
\baselineskip\z@skip \lineskip\z@skip \lineskiplimit\lineskip
\ialign\bgroup\hfil$\m@th\scriptstyle##$\hfil&&\thickspace\hfil
$\m@th\scriptstyle##$\hfil\crcr
\crcr\noalign{\vskip -.3\ex@}%
}{\crcr\noalign{\vskip -.2\ex@}%
\crcr\egroup\egroup\,%
}

\makeatother

\begin{document}

\title{Crouzeix's conjecture holds for tridiagonal $3\times 3$ matrices with elliptic numerical range centered at an eigenvalue}
 
\author{Christer Glader\thanks{\AA bo Akademi University, \AA bo, Finland},~ Mikael Kurula\footnotemark[1] \thanks{Corresponding author, \tt mkurula@abo.fi},~ and Mikael Lindstr\"om\footnotemark[1]}

\thispagestyle{empty}

\maketitle

\begin{abstract} 
M.\ Crouzeix formulated the following conjecture in (Integral Equations Operator Theory 48, 2004, 461--477): For every square matrix $A$ and every polynomial $p$, 
$$
	\|p(A)\| \le 2 \max_{z\in W(A)}|p(z)|,
$$
where  $W(A)$ is the numerical range of $A$. We show that the conjecture holds in its strong, completely bounded form, i.e., where $p$ above is allowed to be any matrix-valued polynomial, for all tridiagonal $3\times 3$ matrices with constant main diagonal:
$$
  \left[\begin{matrix}a&b_1&0\\c_1&a&b_2\\0&c_2&a\end{matrix}\right],\qquad a,b_k,c_k\in\mathbb C,
$$
or equivalently, for all complex $3\times 3$ matrices with elliptic numerical range and one eigenvalue at the center of the ellipse. We also extend the main result of D.\ Choi in (Linear Algebra Appl. 438, 3247--3257) slightly.
\end{abstract}


\noindent
{\bf Keywords:} Crouzeix's conjecture, $3\times 3$ matrix, elliptic numerical range

\kern 3mm

\noindent
{\bf AMS MSC:}15A60, 15A45, 15A18

\section{Introduction}

In \cite{Cro04}, M.\ Crouzeix made the conjecture, henceforth denoted by $(C)$, that for every square matrix $A$ and every
polynomial $p$,
\begin{equation}\label{eq:Cineq}
	\|p(A)\| \le 2 \max_{z\in W(A)}|p(z)|,
\end{equation}
where  $W(A)=\{ x^*Ax:  x^*x =1\}$ is the numerical range of $A$ and $\|\cdot\|$ denotes the appropriate operator 2-norm. M.\ Crouzeix proved in \cite{Cro07} that even if $A$ is an arbitrary bounded linear operator on a Hilbert space, (C) holds if the bound 2 is replaced by 11.08, and very recently Crouzeix and Palencia proved in \cite{CrPa17} that this bound can be drastically reduced to $1+\sqrt2$.

However, proving or disproving (C) in its general formulation, with the bound 2, has turned out to be very challenging; therefore partial results even on seemingly very limited classes of matrices are interesting. Crouzeix \cite{Cro04} showed that the conjecture holds for every $2\times2$ matrix. Later, in \cite{CG12} Greenbaum and Choi  proved the conjecture for Jordan blocks of arbitrary size such that the lower left entry $0$ is replaced by an arbitrary scalar, and this was later extended to a wider class in \cite{Cho13}. In the appendix, we prove the following slight extension of \cite{Cho13} which was stated without proof in \cite{Cro16}:

\begin{observation}\label{obs:ChoiExt} 
(C) holds for all matrices that can be written as $a I+DP$ or $a I+PD$, where $a\in\C$, $D$ is a diagonal matrix, and $P$ is a permutation matrix.
\end{observation}

\noindent 
Our contribution in Observation \ref{obs:ChoiExt} is that we allow permutations with multiple cycles, whereas \cite{Cho13} focuses on the single-cycle case.

In this article, our principal aim is to prove the following result: 

\begin{theorem}\label{mainthm}
Crouzeix's conjecture (C) holds for all matrices of the form
\begin{equation}\label{eq:genclass}
	\bbm{a&b_1&0\\c_1&a&b_2\\0&c_2&a}, \quad a,b_k,c_k\in\C,
\end{equation}
or equivalently, for all complex $3\times3$ matrices with elliptic numerical range centered at an eigenvalue. 
\end{theorem}
 
The reader is kindly referred to the dissertation \cite{ChoiThesis} for a nice overview of Crouzeix' conjecture and a toolbox for studying the conjecture. The recent survey \cite{Cro16} describes the situation a good ten years after (C) was originally formulated, along with some ideas for numerical experiments. For general background on numerical range we refer to \cite{GR97}, and for the particular case of elliptic numerical range to \cite{BS04,KRS97} or the fundamental work \cite{KipEN08}. 

One of the most successful approaches to establishing (C) for a certain class of matrices is by using von Neumann's inequality \cite{Neu50}, and that is the main approach that we use in this paper, too: Let $\mathbb D$ be the open unit disk in $\mathbb C$ and let $A$ be a square matrix which is a contraction in the operator 2-norm. Then von Neumann's inequality asserts that
\begin{equation}\label{eq:vN}
	\|g(A)\|\le \max_{z\in\overline{\mathbb D}}|g(z)|
\end{equation}
holds for all analytic functions $g: \mathbb D \to \mathbb C$ which are continuous (and extended) up to $\partial\mathbb D.$  Now suppose that $A$ is any square matrix and denote its numerical range by $W(A)$. In order to apply \eqref{eq:vN} to $A$, we need a bijective conformal mapping $f$ from the interior of $W(A)$ to $\mathbb D$; this conformal map $f$ can then be extended to a homeomorphism of $W(A)$ onto $\overline{\mathbb D}$. Next we need to find an invertible matrix $X$, such that the similarity transform $X\,f(A)\;X^{-1}$ is a contraction and the condition number of $X$ satisfies
$$
	\kappa(X):=\|X\|\cdot \|X^{-1}\|\le 2.
$$
Then, by von Neumann's inequality, we get for any polynomial
$p$ that
\begin{equation}\label{eq:vNfit}
\begin{aligned}
	\|p(A)\| &= \|X^{-1}\big(p\circ f^{-1}(Xf(A)X^{-1})\big)X\| \\
	&\le 2 \max_{z\in{\overline{\mathbb D}}}|p\circ f^{-1}(z)|= 2 \max_{z\in W(A)}|p(z)|,
\end{aligned}
\end{equation}
which establishes \eqref{eq:Cineq} for $A$. This argument can be shown to establish also the completely bounded version of \eqref{eq:Cineq}. A result by Paulsen \cite[Thm 9.11]{Pau02} shows that this approach using von Neumann's inequality yields the best bound for the completely bounded case.

The paper is organized as follows: In Section 2 we parametrize the matrices in Theorem \ref{mainthm} using two parameters, we parametrize a sufficiently large class of similarity transformations $X$, and we discuss a conformal mapping from a normalized elliptic numerical range onto the unit disk. Finally, in Section 3, we prove Theorem \ref{mainthm} by considering four different cases.

We used Maxima \cite{maxima} and Mathematica \cite{mathematica} for the more involved algebraic calculations.

\section{Parametrization of the matrices $A$ and $X$, and a conformal mapping}

Theorem \ref{mainthm} can be proved by considering a family of matrices, parametrized on a semi-infinite strip in the first quadrant:

\begin{lemma}\label{lem:reduction}
The following statements are equivalent:
\begin{enumerate}
\item (C) holds for the family \eqref{eq:genclass}.
\item (C) holds for all complex $3\times3$ matrices with elliptic numerical range centered at an eigenvalue.
\item (C) holds for the family
\begin{equation}\label{eq:Adef}
	A:=\bbm{1&q/r&r^2-1/r^2\\0&0&qr\\0&0&-1},\qquad q>0,~0<r\leq1.
\end{equation}
\end{enumerate}

The matrix \eqref{eq:Adef} has $\sigma(A)=\{-1,0,1\}$ and $W(A)$ is an ellipse with foci $\pm1$. Defining \begin{equation}\label{eq:rhodef}
	\mu:=(r^2+1/r^2)^2+q^2(r^2+1/r^2)-2\qquad \text{and}\qquad
	\rho:=\sqrt{\frac{\mu+\sqrt{\mu^2-4}}2},
\end{equation}
we have that the major and minor axes of $W(A)$ are $\rho+1/\rho$ and $\rho-1/\rho$, respectively, and $\rho>1$.
\end{lemma}
\begin{proof}
(Step 1: Properties of $A$) We have 
$$
	d:=\frac{q^2}{r^2}+\left(r^2-\frac1{r^2}\right)^2+(qr)^2>0
$$ 
and the number $\lambda\in\sigma(A)$ in \cite[Thm 2.2]{KRS97} is zero, so that $W(A)$ is an ellipse with foci $\pm1$ and minor axis $\sqrt d=\sqrt{\mu-2}$. Moreover, by the definition of $\rho$:
$$
	\rho\geq\sqrt{\frac{\mu-\sqrt{\mu^2-4}}2} = \frac1\rho\qquad\text{and}\qquad
	\left(\rho-\frac1\rho\right)^2=\mu-2,
$$
and so the minor axis has length $\rho-1/\rho$. Then the major axis is 
$$
	\sqrt{(\rho-1/\rho)^2+2^2}=\rho+1/\rho.
$$
By completing some squares in the definition of $\mu$, we see that $\mu\geq 2+2q^2$ which in turn implies that $\mu^2\geq4$ and $\rho>\sqrt{1}$.

(Step 2: 3.\ implies 2.) Assume that (C) holds for the class \eqref{eq:Adef} and let $B\in\C^{3\times3}$ be any matrix with elliptic numerical range and an eigenvalue at the center of $W(A)$. We need to show that (C) holds for $B$ too. One can find $\C\ni a\neq0$ and $b\in\C$ such that $B_1:=aB+b I$ has spectrum $\sigma(B_1)=\{-1,0,1\}$ and the translated, rotated, and scaled numerical range $W(B_1)$ then has its foci at $\pm1$. Next calculate a Schur decomposition of $B_1=QUQ^*$ (with $U$ upper triangular) and select $\theta_1,\theta_2\in\R$ such that $V:=\diag(1,e^{i\theta_1},e^{i\theta_2})$ satisfies
$$
	B_2:=VQB_1(VQ)^*=\bbm{1&2\alpha&2\gamma \\ 0&0&2\beta \\ 0&0&-1},\qquad 
		\alpha,\beta\geq0,~\gamma\in\C;
$$
observe that $VQ$ is unitary. If $\alpha=\beta=\gamma=0$ then $B_2$ is diagonal and it satisfies \eqref{eq:Cineq} in the following stronger form: for all polynomials $p$,
\begin{equation}\label{eq:diagonalest}
	\|p(B_2)\|=\|\diag\big(p(1),p(0),p(-1)\big)\| = \max_{z\in\sigma(B_2)}|p(z)|
	\leq \max_{z\in W(B_2)}|p(z)|.
\end{equation}
This implies that $B$ satisfies the same inequality, i.e., $B$ satisfies (C). It remains to study the case $d:=4 (\alpha^2+\beta^2+|\gamma|^2)>0$. 

By \cite[Thm 2.2]{KRS97}, the assumed properties of $B$ imply that
\begin{equation}\label{eq:param}
	2\alpha\overline\gamma\beta=\beta^2-\alpha^2
\end{equation}
and that the minor axis of $W(B_2)$ has length $\sqrt d$. Next we define $q:=2\sqrt{\alpha\beta}\geq0$. By \eqref{eq:param}, if $q=0$ then $\alpha=\beta=0$. In this case, the unitary matrix
$$
	Y:=\bbm{1&0&0\\0&0&1\\0&1&0}\qquad\text{satisfies}\qquad
	B_3:=YB_2Y^*=\bbm{1&2\gamma&0\\0&-1&0\\0&0&0},
$$
and by \cite[Thm 1.1]{Cro04}, (C) holds for $B_3$, hence also for $B$. For the rest of the proof, we assume that $q>0$.

In the case $q>0$, \eqref{eq:param} implies that $\gamma\in\R$ and then $r:=\sqrt{\sqrt{1+\gamma^2}+\gamma}$ is positive; one easily verifies that $r^2-1/r^2=2\gamma$ and $r^2+1/r^2=\sqrt{4+4\gamma^2}$. Moreover, by \eqref{eq:param}:
$$
\begin{aligned}
	\alpha^2=\beta^2-2\alpha\gamma\beta = (\beta-\alpha\gamma)^2-\alpha^2\gamma^2 \qquad&\implies \\
	 4(\beta-\alpha\gamma)^2 = \alpha^2(4+4\gamma^2) = \alpha^2(r^2+1/r^2)^2 \qquad&\implies \\
	r^2+1/r^2 = |2\beta/\alpha-r^2+1/r^2|\qquad&\implies\qquad 2\beta = \pm 2 \alpha r^{\pm2}.
\end{aligned}
$$
Using $\beta,\alpha\geq0$, $r>0$, and the definition of $q$, we get $4\beta^2=q^2r^2$, i.e., $2\beta=qr$; then $2\alpha=q^2/2\beta=q/r$. If $r\leq 1$ then $B_2$ is of the form \eqref{eq:Adef} and it thus satisfies (C) by the working assumption. 

If $r>1$ then $B_2$ is nevertheless equal to $A$ in \eqref{eq:Adef}. Defining $A'$ to be the matrix in \eqref{eq:Adef} but with $r>1$ replaced by $1/r<1$, we get that (C) holds for $A'$. Moreover, the unitary matrix
$$
	Z:=\bbm{0&0&1\\0&-1&0\\1&0&0}\qquad\text{satisfies}\qquad ZA'Z^*=-A^*;
$$
in particular, $W(A')=-\overline{W(A)}$, where the bar denotes complex conjugate. (This set in fact equals $W(A)$.) Pick an arbitrary polynomial $p$ and define the mirrored conjugate polynomial $\widetilde p(z):=\overline{(p(-\overline z))}$. Using the invariance of the operator norm under adjoints, we then obtain
$$
	\|p(A)\|=\|\widetilde p(-A^*)\|=\|\widetilde p(A')\|\leq2\max_{z\in W(A')}|\widetilde p(z)|
		=2\max_{-\overline z\in W(A)}|p(-\overline z)|.
$$

(Step 3: 2.\ implies 1.) Assume that 2.\ holds and denote the matrix in \eqref{eq:genclass} by $C$; its spectrum is
$$
	\sigma(C)=\set{a\pm\sqrt{b_1c_1+b_2c_2},a}.
$$
By \cite[Thm 4.2]{BS04}, $W(C)$ is an ellipse, and it is straightforward to verify that the number $\lambda$ in \cite[Thm 2.3]{KRS97} equals $a$; hence the foci of $W(C)$ are $a\pm\sqrt{b_1c_1+b_2c_2}$ whose arithmetic average is $a$ by that result.

(Step 4: 1.\ implies 3.) Fix $q>0$ and $r\in(0,1]$. The matrix $A$ has Schur decomposition
$$
	A=Q^*\bbm{0&\displaystyle\frac{\sqrt{1+q^2r^2}}{r^2}&0 \\ 
		\displaystyle\frac{r^2}{\sqrt{1+q^2r^2}}&0&\displaystyle\frac{qr\sqrt{1+q^2r^2+r^4}}{\sqrt{1+q^2r^2}} \\ 
		0&0&0} Q,
$$
where $Q$ is the real, unitary matrix
$$
	\bbm{ \displaystyle\frac{\sqrt{1+q^2r^2}}{\sqrt{r^4+q^2r^2+1}} 
		& \displaystyle-\frac{qr^3}{\sqrt{1+q^2r^2}\sqrt{r^4+q^2r^2+1}} 
		& \displaystyle\frac{r^2}{\sqrt{1+q^2r^2}\sqrt{r^4+q^2r^2+1}} \\
	\displaystyle\frac{r^2}{\sqrt{r^4+q^2r^2+1}} & \displaystyle\frac{qr}{\sqrt{r^4+q^2r^2+1}} 
		&\displaystyle -\frac{1}{\sqrt{r^4+q^2r^2+1}} \\
	 0 & \displaystyle\frac{1}{\sqrt{1+q^2r^2}} & \displaystyle\frac{qr}{\sqrt{1+q^2r^2}}}.
$$
Since (C) is invariant under unitary similarity, the proof is complete.
\end{proof}

In the rest of the paper, we restrict our attention to the matrix family $A$ in \eqref{eq:Adef}. 

\begin{remark}
We shall frequently use $\rho$ as a parameter instead of $q$ as follows: solving \eqref{eq:rhodef} for $q^2$, we get
\begin{equation}\label{eq:qofrho}
	q^2 = \frac{\big(\rho+1/\rho\big)^2-\big(r^2+1/r^2\big)^2}{r^2+1/r^2},
		\qquad \rho>1,~1/\sqrt\rho<r\leq 1\,;
\end{equation}
indeed $q^2$ is increasing in $r\in(0,1]$, so that $q>0$ if and only if $r>1/\sqrt\rho$. For this choice of $q$, \eqref{eq:Adef} becomes
\begin{equation}\label{eq:Arho}
	A=\bbm{1 & \displaystyle\frac{\sqrt{r^4\rho^4-\left(1+r^8\right) \rho^2+r^4}}
		{r^2\sqrt{r^4+1}\,\rho} & r^2-1/r^2\\
	 0 & 0 &\displaystyle\frac{\sqrt{r^4\rho^4-\left(1+r^8\right) \rho^2+r^4}}
		{\sqrt{r^4+1}\,\rho} \\ 
	0 & 0 & -1}.
\end{equation}
An advantage of this parametrization is that $W(A)$ is uniquely determined by the parameter $\rho$, hence independent of $r$.

Moreover, it is often also possible to consider the expression $r^2+1/r^2$ instead of $r$. The advantage of this is that the degree of some polynomials drop to a quarter. We observe that $r^2+1/r^2$ is decreasing with $r\in(0,1]$, that $0<r\leq 1/\sqrt2$ if and only if $r^2+1/r^2\geq 5/2$, and that $1/\sqrt2\leq r\leq1$ if and only if $2\leq r^2+1/r^2\leq 5/2$.
\end{remark}

With $T_{2n}$ denoting the $2n$:th Chebyshev polynomial of the first kind, the mapping 
\begin{equation}\label{eq:HenRho}
	f(z) = \frac{2 z}{\rho} \exp \left( \sum_{n=1}^\infty \frac{2\,(-1)^n\, T_{2n}(z)}{n\,(1+\rho^{4n})} \right)
\end{equation}
in \cite[pp. 373--374]{Hen86} maps the interior of the ellipse with foci $\pm1$ and axes $\rho\pm1/\rho$, $\rho>1$, conformally onto $\D$, with a continuous extension to all of $W(A)$.\footnote{We point out that a number two is missing in \cite{Hen86}.} Moreover, $f(A)$ is computed as follows:

\begin{lemma}\label{lem:geom}
The function $f$ in \eqref{eq:HenRho} satisfies $f(A)=cA$, where
\begin{equation}\label{eq:HenProd}
	c=f(1)= \frac2\rho \prod_{n=1}^\infty \left(\frac{1+\rho^{-8n}}{1+\rho^{4-8n}}\right)^2 <1\, .
\end{equation}
Furthermore, we have the estimates $c < \frac{2}{\rho}$ for $\rho > 1$ and $c < \frac{2}{\rho\sqrt{1+4/\rho^{4}}}$ for $\rho \ge \sqrt{2}$. 
\end{lemma}

\begin{proof} 
The point $1$ is in the interior of $W(A)$, and therefore $f(1)\in\mathbb D$, which implies that the real number $f(1)$ is less than one. Since $A$ has $\sigma(A)=\set{-1,0,1}$, there exists a matrix $Z$ of eigenvectors such that $A=Z\diag(1,0,-1)  Z^{-1}$ and since $\sigma(A)\subset\mathrm{int}\, W(A)$ and $T_{2n}(-1)=T_{2n}(1) = 1$,
$$
	f(A)=Z\,f\big(\diag(1,0,-1)\big)Z^{-1}=Z\diag(c,0,-c)Z^{-1}=cA;
$$
see e.g.\ \cite[\S6.2]{HoJoBook}.

Next we show that for $\rho>1$, it holds that 
\begin{equation}\label{eq:logseries}
	\sum_{n=1}^\infty \frac{2\,(-1)^n}{n\,(1+\rho^{4n})}
	=  2 \sum_{n=1}^\infty (-1)^{n} \ln\left(1 + \rho^{-4n}\right).
\end{equation}
We have $\rho^{-4n}\in(0,1)$, so by the formula for a geometric series,
$$
\begin{aligned}
	\sum_{n=1}^\infty \frac{2\,(-1)^n}{n\,(1+\rho^{4n})} &= 
		\sum_{n=1}^\infty \frac{2\,(-1)^n}{n\,\rho^{4n}}\sum_{k=0}^\infty (-1)^k (\rho^{-4n})^k \\
	&=\sum_{n=1}^\infty \sum_{k=0}^\infty \frac{2\,(-1)^n}{n} (-1)^k (\rho^{-4n})^{k+1}.
\end{aligned}
$$
Since $\sum_{n=1}^\infty \sum_{k=0}^\infty \frac{1}{n} (\rho^{-4n})^{k+1} < \infty$, we can interchange the order of the summation in the above double series in order to obtain
$$
	\sum_{n=1}^\infty \frac{2\,(-1)^n}{n\,(1+\rho^{4n})}
	=\sum_{k=0}^\infty \sum_{n=1}^\infty \frac{2\,(-1)^n}{n}(-1)^k\left(\rho^{-4(k+1)}\right)^{n}.
$$ 
Now we use that 
$$
	\ln\left(1 + \rho^{-4(k+1)}\right) = 
	\sum_{n=1}^\infty \frac{(-1)^{n+1}}{n} \left(\rho^{-4(k+1)}\right)^n,
$$ 
and \eqref{eq:logseries} follows. As a consequence, the number $c(\rho)=f(1)$ in \eqref{eq:HenRho} can be written
$$
	c(\rho)=\frac2\rho \exp
		\left(2\lim_{N\to\infty} \sum_{n=1}^N (-1)^n\,\ln\left(1 + \rho^{-4n}\right)\right)
	=  \frac2\rho \lim_{N\to\infty}\prod_{n=1}^N \left(\frac{1+\rho^{-8n}}{1+\rho^{4-8n}}\right)^2,
$$
where we used the continuity of the exponential function, and this establishes \eqref{eq:HenProd}.

Clearly, the series on the RHS of \eqref{eq:logseries} is alternating, with a negative first term. Moreover, since $0<\rho^{-4}<1$, the sequence $0<\rho^{-4n}<1$ is decreasing with $n$, and then, so is $\ln\left(1+\rho^{-4n}\right)>0$. Hence, any truncation of the series on the RHS of \eqref{eq:logseries} to an even number of terms gives a strict overestimate of the series, and combining this with the monotonicity of the exponential function, we obtain that all truncations of the product in \eqref{eq:HenProd} yield strict overestimates of $f(1)$. We obtain the overestimate $c<2/\rho$ by truncating to zero factors and by truncating to two factors, we obtain that for all $\rho>1$:
$$
	c<\frac{2\, (1+\rho^{-8})^2\,(1+\rho^{-16})^2}{\rho\cdot(1+\rho^{-4})^2\,(1+\rho^{-12})^2}.
$$
Thus, the proof is complete once we establish that
$$ 
\frac{2 (1 + \rho^{-8})^2 (1 + \rho^{-16})^2}{\rho (1 + \rho^{-4})^2 (1 + \rho^{-12})^2}\le  \frac{2}{\rho\sqrt{1+4\rho^{-4}}}
$$
for $\rho \ge \sqrt{2}$. This amounts to proving, with $t:=\rho^4\ge 4$, that
$$
\begin{aligned}
        \frac{(4 +  \rho^{4}) (1 +  \rho^{8})^4 (1 +  \rho^{16})^4}{\rho ^{36} (1 +  \rho^{4})^4 (1 +  \rho^{12})^4 }\leq 1 \iff 
         \frac{(4 + t) (1 + t^2)^4 (1 + t^4)^4}{ t^9 (1 + t)^4 (1 + t^3)^4 }&\leq 1\\ \iff 
	 q(t) := (4 + t) (1 + t^2)^4 (1 + t^4)^4 -  t^9 (1 + t)^4 (1 + t^3)^4 &\leq 0\,. 
\end{aligned}
$$
To verify the last inequality we expand $q(t)$ and make some further estimates:
$$
\begin{aligned}
q(t) &=4 + t + 16 t^2 + 4 t^3 + 40 t^4 + 10 t^5 + 80 t^6 + 20 t^7 + 124 t^8 \\
       &\qquad+ 30 t^9 + 156 t^{10} + 34 t^{11} + 168 t^{12} + 27 t^{13} + 136 t^{14} \\
       & \qquad+ 18 t^{15} + 96 t^{16} - 5 t^{17} + 52 t^{18} - 2 t^{19} + 16 t^{20} \\ 
       & \qquad - 7 t^{21} + 8 t^{22} - 2 t^{23}\ ,
\end{aligned}
$$
and using $t\ge 4$, we have
$$
\begin{aligned}
q(t) \le\, & (4 + 1 + 16 + 4 +\, \cdots\,  + 27 + 136 + 18 + 96) t^{16}\\
         &- 5 t^{17} + 52 t^{18} - 2 t^{19} + 16 t^{20} - 7\cdot 4\cdot t^{20} + 8 t^{22} - 2\cdot 4\cdot t^{22}\\
      =\, &964 t^{16} - 5 t^{17} + 52 t^{18} - 2 t^{19} - 12 t^{20}\\
      \le\, &964 t^{16} - 5 t^{17} + 52 t^{18} - 2 t^{19} - 12\cdot4\cdot4\cdot t^{18}\\
      =\, &964 t^{16} - 5 t^{17} - 140 t^{18} - 2 t^{19}\le 964 t^{16} - 5 t^{17} - 140\cdot 4\cdot 4\cdot t^{16} - 2 t^{19}\\
      =\, &-1276 t^{16} - 5 t^{17} - 2 t^{19} < 0\, .
\end{aligned}
$$

\end{proof}

Now, for similarity transformations with $\kappa(X)\leq2$, we are able to apply \eqref{eq:vNfit} if and only if
\begin{equation}\label{eq:crit}
	\|Xf(A)X^{-1}\|=c\,\|XAX^{-1}\|\leq 1,
\end{equation}
and we shall verify this inequality on a large portion of the semi-infinite strip in \eqref{eq:Adef}. 

We finish this section by parameterizing a large class of upper triangular $3\times 3$ matrices $X$ with $\kappa(X)=2$. We were led to this class of similarity transformations by solving the optimization problem
\begin{equation}\label{eq:optprop}
	\mathrm{arg\,min}_X\ \|XAX^{-1}\|\qquad \text{such that}\qquad \kappa(X)\leq 2
\end{equation}
numerically on grids covering various parts of the semi-strip in \eqref{eq:Adef}.

\begin{proposition}\label{prop:Xparam}
The matrix
$$
	X=\bbm{s&t&u\\0&1&v\\0&0&w},\qquad s,t,u,v,w\in\R,
$$
has a singular value equal to $1$ if and only if
\begin{equation}\label{eq:kappa1}
	2tuv=s^2v^2-t^2+t^2v^2+t^2w^2-v^2\, .
\end{equation} 
Furthermore, $X$ is invertible with $\kappa(X)=2$, provided that $1/2\leq sw\leq2$, 
and that
\begin{equation}\label{eq:kappa2}
	\frac52sw=s^2+t^2+u^2+v^2+w^2.
\end{equation}
\end{proposition}

\begin{proof}
The characteristic polynomial of $X^*X$ is
$$
\begin{aligned}
	p_X(\lambda):=&\lambda^3-(s^2+t^2+u^2+v^2+w^2+1)\lambda^2 \\
		&\quad +\big(s^2-2tuv+u^2+(s^2+t^2)v^2+(s^2+t^2+1)w^2\big)\lambda-s^2w^2
\end{aligned}
$$
and using
$$
	p_X(1)=s^2v^2-t^2-2tuv+t^2v^2+t^2w^2-v^2 = 0,
$$
we get $p_X(\lambda)=(\lambda-1)(\lambda^2-\xi\lambda+\eta)$, where
$$
	\xi:=s^2+t^2+u^2+v^2+w^2\qquad \text{and}\qquad
	\eta:=s^2w^2.
$$
Hence, the squared singular values of $X$ are
$$
	\sigma_1^2=1,\qquad \sigma_-^2:=\frac{\xi-\sqrt{\xi^2-4\eta}}2,
	\qquad\text{and}\qquad \sigma_+^2:=\frac{\xi+\sqrt{\xi^2-4\eta}}2.
$$
The assumption $\xi=\frac52\sqrt\eta$ implies that $\sigma_+^2=4\sigma_-^2$, and this in turn implies that
$$
	\sigma_+^2=2\sigma_-\sigma_+=2\sqrt\eta=2sw,
$$
so that $\sigma_-^2=sw/2$. Thus, $\kappa(X)^2=\sigma_+^2/\sigma_-^2=4$ if $\sigma_-\leq1\leq\sigma_+$, which is true because it is equivalent to the assumption $1/2\leq sw\leq 2$.
\end{proof}

\begin{remark}\label{rem:normest}
If $s w\ne 0$ in Prop. \ref{prop:Xparam} then $X$ is invertible and $XAX^{-1}$ equals the matrix
$$
	G:=\bbm{1&\alpha&\gamma\\0&0&\beta\\0&0&-1},
$$
where 
$$
 \alpha = \frac{q s-rt}{r}\, ,\quad \beta =  \frac{qr-v}{w},\quad \text{and}\quad
   \gamma =  \frac{qr^3t-qrsv+r^4s+r^2tv-2r^2u-s}{r^2w}\, .
$$
Calculation of the singular values gives that $\|G\|^2$ is the larger zero of the polynomial
\begin{equation}\label{CanP}
P(\lambda) := \lambda^2 - (2+\alpha^2+\beta^2+\gamma^2) \lambda +1+\alpha^2+\beta^2+\alpha^2 \beta^2\, ,
\end{equation}
and since $P(\lambda)$ is a parabola with minimum at $(2+\alpha^2+\beta^2+\gamma^2)/2$, the following useful equivalence follows
\begin{equation}\label{CanP2}
  || XAX^{-1} || \le \mu\qquad\iff\qquad P(\mu^2)\ge 0\quad\text{and}\quad 2+\alpha^2+\beta^2+\gamma^2 \le 2   \mu^2\, .
\end{equation}
The condition $2+\alpha^2+\beta^2+\gamma^2 \le 2   \mu^2$ can be replaced by the condition $P'(\mu^2)\geq0$.
\end{remark} 

Using the family of similarity transformations in Prop. \ref{prop:Xparam}, we found the rather optimal transforms $X$ presented in the next section via elementary optimization by hand combined with structural clues provided by the numerical solution of \eqref{eq:optprop}.

\section{The proof of Theorem \ref{mainthm}}

Different similarity families $X$ are required for different choices of the parameters $(\rho,r)$; the situation is illustrated in Fig.\ \ref{fig:X} below. We begin the study with the case where roughly $r\leq 1/\sqrt 2$: 

\begin{figure}[!h]
\begin{center}
\includegraphics[width=0.7\textwidth]{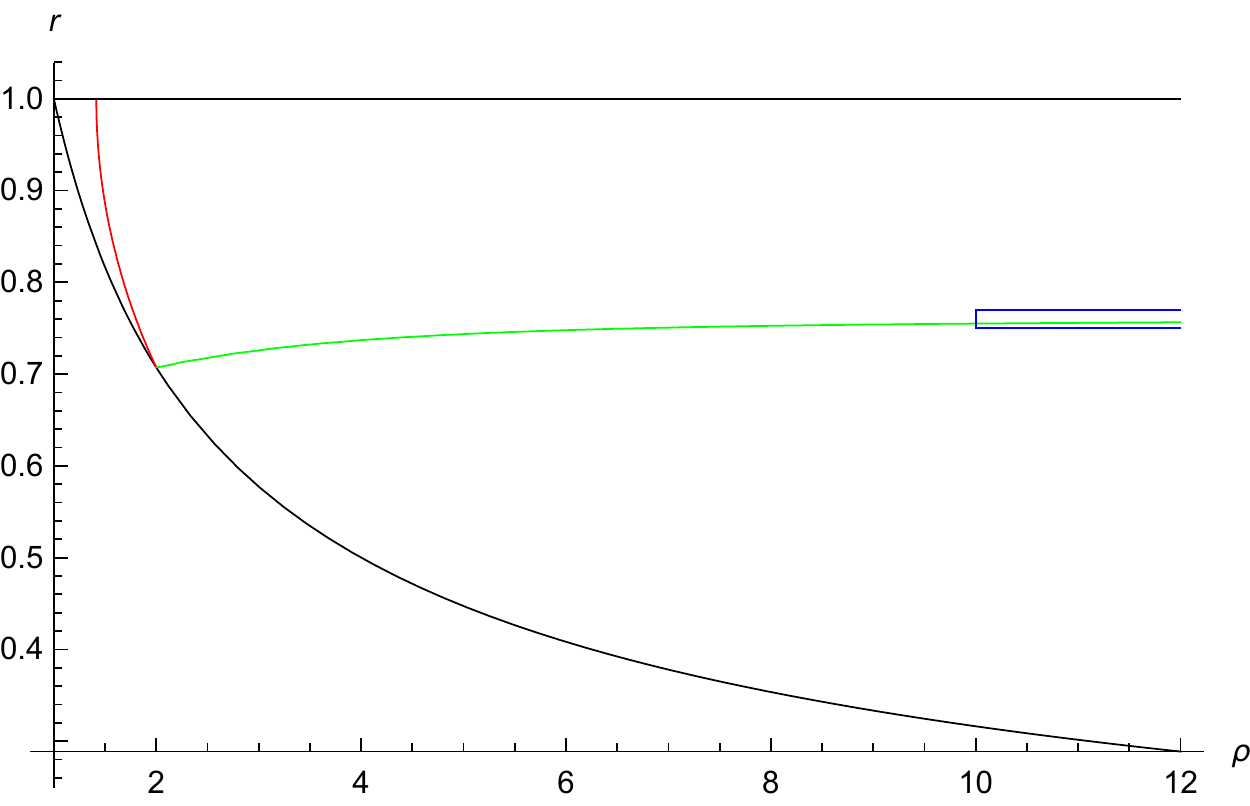}
\caption{The area $1/\sqrt\rho<r\leq 1$ is bounded by the black lines. We show that (C) holds for $(\rho,r)$ under the green curve in \S\ref{sec:smallr}; after that we will add the narrow semi-strip that starts at $\rho=10$, in \S\ref{sec:semistrip}. In section \S\ref{sec:diagnew}, we do the case under the red curve, and in \S\ref{sec:final}, we finally prove that (C) holds for the remaining values of $(\rho,r)$, i.e., the area containing the upper-right corner of the picture.}
\label{fig:X}
\end{center}
\end{figure}

\subsection{The case of small $r$}\label{sec:smallr}

The following gives (C) for a large subset of $(\rho,r)$:

\begin{proposition}\label{prop:smallr}
The conjecture (C) holds for $A$ provided that $(\rho,r)$ satisfies $1/\sqrt\rho<r\leq r_1(\rho)$, where 
$r_1(\rho)$ is the unique positive root $r$ of the equation ($\rho>1$ is fixed)
\begin{equation}\label{eq:psmallr}
p(r,\rho) := 12 r^8 + r^4 \left(3 \rho^2+16 + \frac{4}{\rho^2}\right) - 4 - \rho^2 = 0\, .
\end{equation}
The positive zero $r_1(\rho)$ of $p(\cdot,\rho)$ depends continuously on, and increases strictly with, $\rho$; moreover $r_1(\rho)<1/\sqrt[4]3$ for all $\rho>1$.
\end{proposition}

\begin{proof} 
Using the matrix $X$ in Prop. \ref{prop:Xparam} with the parameters 
$$
\begin{aligned}
v &= 0\, ,\quad w=1\, , \quad s=\frac{1+q^2r^2+4r^4}{2q^2r^2+2r^4+2}\, ,\\
 t&=\frac{(2s-1)q}{2r}\, \quad \text{and}\quad u=-\frac{\sqrt{(2-s)( 2s-1)-2t^2}}{\sqrt{2}}\, ,
\end{aligned}	
$$
the reader may verify that $X$ meets the sufficient conditions in Prop.\ \ref{prop:Xparam}; hence $\kappa(X)=2$. Moreover,
$$
	XAX^{-1}=\bbm{1 &\displaystyle \frac{q}{2r} &\displaystyle\frac{4r^4-1}{2r^2} \\ 
		0 & 0 & qr \\ 0 & 0 & -1}\, ,
$$
for which $P(\lambda)$ in (\ref{CanP}) factors into 
$$
P(\lambda) = \left(\lambda - \frac{1+q^2 r^2}{4 r^4}\right)\left(\lambda - 4 r^4-q^2 r^2\right).
$$
Now we show that the RHS of (\ref{CanP2}) holds with the choice $\mu = \rho/2$ , which implies 
that $ \|XAX^{-1}\|\le \rho/2$ and hence $c(\rho) \|XAX^{-1}\| < 1$.  To this end we note that for $r\geq 1/\sqrt\rho$:
$$
\begin{aligned}
 \frac{\rho^2}{4} - \frac{1+q^2 r^2}{4 r^4} &=\frac{(1+\rho^2)(\rho r^2-1)(\rho r^2+1)}{4\rho^2(r^4+1)} 	>0 \qquad \text{and}\\
\frac{\rho^2}{4} - 4 r^4-q^2 r^2 &= 
 \frac{4+\rho^2-r^4 (3 \rho^2+16+4/\rho^2)-12r^8}{4 (r^4+1)}\, ;
\end{aligned}
$$
observe that the latter numerator is $-p(r,\rho)$. We have that $\partial p/\partial r>0$ for all $r>0$ and then it follows from $0 < r \le r_1(\rho)$ that $p(r,\rho) \le 0$, so that $P(\rho^2/4)\ge 0$. Moreover, $2 +\alpha^2+\beta^2+\gamma^2\le 2 \rho^2/4$, because
$$
	2 +\alpha^2+\beta^2+\gamma^2-2\frac{ \rho^2}{4} = 
	\frac{p(r,\rho)}{4(r^4+1)}-\frac{(1+\rho^2)(\rho r^2-1)(\rho r^2+1)}{4\rho^2(r^4+1)}\leq0.
$$

Since $p(\cdot,\rho)$ is increasing on $\rplus$, and quadratic in $r^4$, it has at most one zero $r>0$, and considering that $p(0,\rho)<0$, we obtain that $p$ has exactly one positive zero. Finally, for $r<1/\sqrt[4]3$, it holds that
$$
	\frac{\partial p}{\partial\rho}=-2\rho(1-3r^4)-8r^4/\rho^3<0,
$$
and implicit differentiation gives
$$
	0=\frac{\mathrm d}{\mathrm d\rho} \, p\big(r_1(\rho),\rho\big)=
	\frac{\partial p}{\partial r}\big(r_1(\rho),\rho\big) \cdot r_1'(\rho)
		+\frac{\partial p}{\partial \rho}\big(r_1(\rho),\rho\big),
$$
so that for $r_1(\rho)<1/\sqrt[4]3$,
$$
	 r_1'(\rho) = -\frac{\frac{\partial p}{\partial \rho}\big(r_1(\rho),\rho\big)}
	 	{\frac{\partial p}{\partial r}\big(r_1(\rho),\rho\big)}>0.
$$
In particular, $r_1(\rho)$ is continuous and strictly increasing with $\rho$ as long as $r_1(\rho)<1/\sqrt[4]3$. Moreover, $p(1/\sqrt2,2)=0$, so that $r_1(2)=1/\sqrt[4]4$, and no $\rho>1$ exists such that $r_1(\rho)=1/\sqrt[4]3$, because
$$
	p\left(\frac1{\sqrt[4]3},\rho\right)=\frac{4(1+2\rho^2)}{3\rho^2}\neq 0.
$$
\end{proof}

We mention, but do not build on, the fact that $r_1(\rho)$ in Prop.\ \ref{prop:smallr} increases from $1/\sqrt{2}$ to $1/\sqrt[4]3$ as $\rho$ runs from $2$ to $\infty$.

\subsection{A small extension for large $\rho$}\label{sec:semistrip}

We next consider points $(\rho,r)$ in the semi-strip $\rho\ge 10,\, 0.75\leq r\leq 0.77$. Here it suffices to take
\begin{equation}\label{eq:zform4}
	X=\bbm{r/\sqrt2&0&0\\0&1&0\\0&0&r\,\sqrt2}
\end{equation}
which clearly has $\kappa(X)=2$. Moreover, the polynomial in (\ref{CanP}) becomes
\begin{equation}\label{CanP3}
P(\lambda) = \lambda^2 - \left(q^2+1+\frac{1}{4}\left(r^2+\frac{1}{r^2}\right)^2\right) \lambda + \left(1+\frac{q^2}{2}\right)^2\, ,
\end{equation} 
with $q^2$ given by \eqref{eq:qofrho}.

\begin{proposition}\label{prop:case4}
The conjecture (C) holds for $A$ when $\rho\geq 10$ and $r\leq 0.77$.
\end{proposition}
\begin{proof}
First we note that with $X$ and $r_1(\rho)$ as in Prop. \ref{prop:smallr}, for $\rho \ge 10$ it holds that $r_1(\rho) \ge r_1(10) > 0.75$, because $p(r,10)=0$ if and only if $300r^8+7901r^4-2600=0$ which is easily solved for $r^4>0.325$. Prop. \ref{prop:smallr} then gives that (C) holds for $\rho\ge 10$ and $1/\sqrt{\rho} < r\le 0.75$.
 
Now we consider the semi-strip where $\rho \ge 10$ and $0.75\leq r\leq 0.77$, with $X$ and $P(\lambda )$ given in \eqref {eq:zform4} and (\ref{CanP3}). We introduce $x := r^2 + 1/r^2\in(2.2795,2.35)$, which is strictly decreasing in $r\in[0,1]$, and $y := \rho + 1/\rho$ (strictly increasing in $\rho>1$). Then $10 < y \le 1.01\, \rho$ for $\rho\ge 10$ and $q^2 = y^2/x - x$ by \eqref{eq:qofrho}. 

Our plan is to verify the conditions in \eqref{CanP2} for $\mu := y/2.02$; then by Lemma \ref{lem:geom} it holds that
$$
	c\,\|XAX^{-1}\|\leq c\,\frac{y}{2.02}< \frac2\rho \cdot \frac \rho2=1.
$$
The value of
\[
	P(\mu^2)=\left(\frac{y}{2.02}\right)^4
	- \left( \left(\frac x2-1\right)^2+\frac{y^2}{x}\right) \cdot \left(\frac{y}{2.02}\right)^2
	+\frac14\left( 2-x+\frac{y^2}{x}\right) ^2
\]
at $x=2.2795$ and $y=10$ is positive, and we will show that $P(\mu^2)$ is increasing in both $x$ and $y$. Indeed,
\[
\begin{aligned}
2\cdot101^2\, x^3\,\frac{\partial P(\mu^2)}{\partial x} &=
-2\cdot101^2\, x^3+101^2\, x^4-2\cdot101^2\, x\, y^2+5000\, x^3\, y^2 \\
&\qquad -50^2\, x^4\, y^2-101^2\, y^4+5000\, x\, y^4,
\end{aligned}
\]
where the RHS is larger than
\[
\begin{aligned}
-2\cdot101^2\cdot 2.35^3+101^2\cdot 2.2795^4-2\cdot101^2\cdot 2.35\, y^2+5000\cdot 2.2795^3\, y^2 \\
\qquad -50^2\cdot  2.35^4\, y^2-101^2\, y^4+5000\cdot 2.2795\, y^4 &= \\
\frac{2393}{2}\, y^4-\frac{103947096121}{1600000}\, y^2+\frac{170383737131921561}{16000000000000}
\end{aligned}
\]
which is quadratic in $y^2$, and positive with a positive derivative at $y=10$. Hence $\partial P(\mu^2)/\partial x>0$ for the relevant values of $x$ and $y$. By factorizing the partial derivative wrt. $y$, we obtain that $\partial P(\mu^2)/\partial y>0$ if and only if 
$$
	(5000\,x-101^2)^2\,y^2-101^2(x^2-2x)(101^2-50^2x+1250x^2)>0,
$$
which holds, since 
$$
\begin{aligned}
	(5000\,x-101^2)^2\,y^2-101^2(x^2-2x)(101^2-50^2x+1250x^2) &>\\
	(5000\,x-101^2)^2\cdot 10^2-101^2(x^2-2x)(101^2-50^2x+1250x^2) &= \\
	-12751250x^4+51005000x^3+2344934599x^2-9992879198x+10406040100 &=:p(x)
\end{aligned}
$$
satisfies $p'''(x)<0$ for $x>1$, $p''(x)>10^9$ for $x<2.35$, and $p'(x)>10^8$, $p(x)>10^7$ for $x>2.2795$. Hence, $P(\mu^2)>0$ for the considered values of $x$ and $y$.

It remains only to verify the second condition in (\ref{CanP2}); $y^2/x-x+1+x^2/4\le 2 (y/2.02)^2$. This holds for $ 2.2795 < x < 2.35$ and $y>10$, because
  $$
y^2> 100 > \frac{1.01^2\cdot 2.35\cdot (2.35-2)^2}{2\cdot 2.2795-2.02^2} > \frac{1.01^2\, x\, (x-2)^2}{2x-2.02^2}.
  $$ 
  The proof is complete.
 \end{proof}

\subsection{Diagonalization}\label{sec:diagnew}

For most of the cases $r>1/\sqrt2$ we will consider similarities $X$ for which there is some $z\in\R$ such that
\begin{equation}\label{eq:zform}
	XAX^{-1}=\bbm{1&z&0\\0&0&z\\0&0&-1};
\end{equation}
the norm squared of this matrix is $1+z^2$.
It turns out that the matrix 
\begin{equation}\label{eq:Xz}
	X=\bbm{s&\displaystyle\frac{sv}{r^2}&\displaystyle\frac{r^4s+sv^2-s}{2r^2}
		\\ 0&1&v \\ 0&0&\displaystyle\frac{r^2}s},\qquad s>0,
\end{equation}
achieves \eqref{eq:zform} with $z=(qr-v)s/r^2$. In \eqref{eq:Xz}, $w=r^2/s$, so that $ws=r^2\in[1/2,1]$ for $1/\sqrt2\leq r\leq1$, and it is easy to verify that \eqref{eq:kappa1} holds, too. Therefore $\kappa(X)=2$ is guaranteed if \eqref{eq:kappa2} holds. Solving \eqref{eq:kappa2} for $v\geq0$, we find
\begin{equation}\label{eq:v3}
	v=\frac{\sqrt{\sqrt2\,r^2s\,\sqrt{(2r^2+1)(r^2+2)}-r^4s^2-2r^4-s^2}}s;
\end{equation}
hence \eqref{eq:Xz} with this choice of $v$ has $\kappa(X)=2$, assuming that $v\in\R$.

\begin{proposition}\label{prop:diagnew}
The conjecture (C) holds for $(\rho,r)$ satisfying
\begin{equation}\label{eq:diagarearhoprel}
r\geq\frac1{\sqrt\rho}  \qquad\text{and}\qquad 
2\big(\rho+1/\rho\big)^2\le  \left(r^2+\frac1{r^2}\right)^2 + \frac{5}{2} \left(r^2+\frac1{r^2}\right)\, . 
\end{equation}
Moreover, there exists a non-increasing function $r_3(\rho)$ such that one can write \eqref{eq:diagarearhoprel} equivalently as $1/\sqrt\rho\leq r\leq r_3(\rho)$, $1<\rho\leq 2$.
\end{proposition}

\begin{proof}
Assume that $(\rho,r)$ satisfies \eqref{eq:diagarearhoprel}. We can then find a matrix $X$ which diagonalizes $A$ and has $\kappa(X)=2$, and after that we can use a calculation very similar to \eqref{eq:diagonalest} to obtain
$$
	\|p(A)\|\leq \kappa(X)\max_{x\in\sigma(A)}|p(z)|
$$
for all polynomials, also in the completely bounded case. This is convenient, because it spares us from the complication of dealing with $c$. 

We again set $x:=r^2+1/r^2$ and $y:=\rho+1/\rho$ in order to turn the second inequality in \eqref{eq:diagarearhoprel} into 
\begin{equation}\label{eq:diagxy}
	2y^2\leq x^2+\frac52\,x,
\end{equation}
and \eqref{eq:v3} becomes
\begin{equation}\label{eq:v3x}
	v=\frac{\sqrt{\sqrt2\,r^3s\,\sqrt{5+2x}-r^2s^2x-2r^4}}s.
\end{equation}
By \eqref{eq:zform}, $XAX^{-1}$ is diagonal if and only if $qr=v$, and solving \eqref{eq:v3x} equal to $qr>0$ with respect to $s$ gives
\begin{equation}\label{eq:sdiagx}
	  s=r\,\frac{\sqrt{5+2x}-\sqrt{5-2x-4q^2}}{\sqrt2\,(q^2+x)}.
\end{equation}
Solving \eqref{eq:qofrho} for $y^2$, we get $y^2=q^2x+x^2$, and substituting this into \eqref{eq:diagxy}, we get $5-2x \geq 4q^2$; thus $s > 0$. Hence (C) is established using only the second inequality in \eqref{eq:diagarearhoprel}.  

Already at the introduction \eqref{eq:qofrho} of $\rho$, we require that $\rho>1$. Moreover, $r\geq1/{\sqrt\rho}$ translates into $x\leq y$, and then $2y^2\leq x^2+5x/2\leq y^2+5y/2$ implies $y\leq 5/2$, i.e., $\rho\leq 2$. Solving $2y^2= x_0^2+5x_0/2$ for $x_0>0$, we get $x_0=\sqrt{5^2/4^2+2y^2}-5/4$ which increases from $2$ to $5/2$ as $\rho$ increases from $\sqrt2$ to $2$. Solving $r^2+1/r^2=x_0$ for $r$, we get $r_3(\rho):=\sqrt{x_0/2-\sqrt{x_0^2/4-1}}$ which decreases from $1$ to $1/\sqrt2$ as $\rho$ increases from $\sqrt2$ to $2$. For $1<\rho<\sqrt2$, we simply set $r_3(\rho):=1$. This is a non-increasing function $r_3$ such that $1/\sqrt\rho\leq r\leq r_3(\rho)$ and $1<\rho\leq 2$.

Conversely, if $1/\sqrt\rho\leq r\leq r_3(\rho)$ and $1<\rho\leq 2$ then $x\geq x_0$, so that $2y^2= x_0^2+5x_0/2\leq x^2+5x/2$, i.e., \eqref{eq:diagarearhoprel} holds for all $\rho\geq\sqrt 2$. For $1<\rho<\sqrt2$, we have $2y^2<9$ and $r\leq 1$ implies $x\geq2$, so that $x^2+5x/2\geq 9$, i.e., \eqref{eq:diagarearhoprel} holds.
\end{proof}

\subsection{The final case: large $\rho$ and $r$}\label{sec:final}

In \S\ref{sec:diagnew}, we determined $s$ by the condition that $XAX^{-1}$ is diagonal, which restricted us to quite a small region for $(\rho,r)$. Instead of requiring diagonality of the transform, we now choose $s$ as a critical point of $\|XAX^{-1}\|^2$.

\begin{lemma}
At a point $(r,\rho)$ with $1/\sqrt2\leq r\leq 1$it is possible to choose $s>0$ and $v\geq0$ such that $X$ in \eqref{eq:Xz} has $\kappa(X)=2$ and
\begin{equation}\label{eq:psinorm}
	\| X A X^{-1}\|^2 = \psi(x,y) := \frac{10 y^2-5x^2-2y \sqrt{y^2-x^2}\sqrt{25-4x^2}}{2x^3},
\end{equation}
where $x=r^2+1/r^2$ and $y=\rho+1/\rho$.
\end{lemma}

\begin{proof}
From \eqref{eq:v3x} it follows that
\begin{equation}\label{eq:vssq}
	v^2 s^2 = r^3 s\, \sqrt{10+4x}-r^2s^2 x-2 r^4\, ,
\end{equation}
whose partial derivative with respect to $s$ equals $r^3\, \sqrt{10+4x}-2r^2s x$; by the chain rule this is  moreover equal to
$$
	2vs\,\frac{\mathrm d \big(v(s)\cdot s\big)}{\mathrm ds}.
$$
With $z(s)=\big(qr-v(s)\big)s/r^2$, the equation $z'(s) = 0$ implies that $(v(s) s)' = q r$, so that
$$
	r^2\, \sqrt{10+4x}=2qvs+2rs x
$$
at a critical point of $\|XAX^{-1}\|^2$, and from this we get
\begin{equation}\label{eq:vssq2}
	v^2s^2 =\frac{10r^4+4xr^4-4r^3sx\, \sqrt{10+4x}+4r^2s^2 x^2}{4q^2}.
\end{equation}
We wish to eliminate $v$ and solve for $s$. 

By \eqref{eq:vssq}, we have
$$
	4r^3sx\, \sqrt{10+4x} = 4xv^2 s^2 +4r^2s^2x^2+8 r^4x \, ,
$$
and substituting this into \eqref{eq:vssq2}, we obtain that $v\geq0$ is real if $s$ is real, because $1/\sqrt2\leq r\leq 1$ implies that $2\leq x\leq 5/2$ and
$$
	v^2s^2=\frac{(5-2x)r^4}{2q^2+2x}\geq0.
$$
Combining this with \eqref{eq:vssq} gives
$$
 s^2x - rs\, \sqrt{10 +4x} + r^2\Big(\frac{5 -2x}{2x+ 2q^2} +2 \Big)=0
$$
which is a quadratic equation in $s^2$ that no longer contains $v$. The smaller root 
$$
	s= \frac{r \sqrt{5 + 2x}}{\sqrt2\,x} - \frac{rq\,\sqrt{5-2x}}{x \sqrt{2x +2q^2}},
$$
is real and strictly positive since $2\leq x\le 5/2$.

For these choices of $s$ and $v$, $sw=r^2\in[1/2,1]$, so that $\kappa(X)=2$, and
$$
1+z^2 = 1 +\frac{q^2 s^2}{r^2} -2 \frac{q s^2 v}{r^3}+ \frac{s^2 v^2}{r^4} = \psi(x,y).
$$
This completes the proof.
\end{proof}

\begin{proposition}\label{prop:caser=1}
The conjecture (C) holds for $A$ when $r = 1$.
\end{proposition}
\begin{proof} Prop.\ \ref{prop:diagnew} gives that (C) holds for $r=1$ and $1<\rho<\sqrt2$ and by Lemma \ref{lem:geom} we have the estimate $c < \frac{2}{\rho\sqrt{1+4/\rho^{4}}}$ for $\rho \ge \sqrt{2}$. With $r=1$ we have $x=2$ and
 $$
 \psi (2,y) = \frac{5 y^2-10-3 y \sqrt{y^2-4}}{8} = \frac{4+\rho^4}{4 \rho^2}\, ;
 $$
note that $y=\rho+1/\rho$ and $\sqrt{y^2-4}=\rho-1/\rho$. Thus
$c^2 \| X A X^{-1}\|^2 < 1$.
\end{proof}

The rest of the paper is based on the following lemma which together with Prop.\ \ref{prop:caser=1} says that for a fixed $\rho$, it is enough to evaluate \eqref{eq:psinorm} at the point $(\rho,r)$ with the smallest relevant $r$: 

\begin{lemma}\label{lem:lastpatchmono}
With $x=r^2+1/r^{2}$ and $y=\rho+1/\rho$, assume that  $2y\ge \sqrt{5 x+2 x^2}$ (which is the reverse of the second inequality in \eqref {eq:diagarearhoprel}). Then $1/\sqrt2\leq\tau\leq r\leq 1$ implies that
$$
\|X A X^{-1}\|^2\le\max \left(\psi(2,y),\psi(\tau^2+1/\tau^{2},y)\right).
$$
\end{lemma}

\begin{proof}
A calculation shows that 
$$
\begin{aligned}
	2 x^4 \sqrt{y^2-x^2}&\sqrt{25-4 x^2}\, \frac{\partial\psi}{\partial x}(x,y) = b-a \\
	\text{with}\qquad a&=(30 y^2-5 x^2) \sqrt{y^2-x^2}\sqrt{25-4 x^2}\\ 
	\text{and}\qquad b&=2 y (75 y^2-8x^2y^2-50 x^2+4x^4)\, .
\end{aligned}
$$
From $2\le x\le 5/2$ and $y\ge\sqrt{5 x+2 x^2}/2 \geq x$ it follows that $a\geq0$ and
$$
	b \geq 2y\left( (75-8x^2)\frac{5x+2x^2}{4}-50x^2+4x^4 \right)\geq0,
$$
with strict inequalities if $2\leq x<5/2$, i.e., $1/\sqrt2<r\leq 1$. Then $b+a>0$ and $\frac{\partial\psi}{\partial x}(x,y)$ has the same sign as $b^2-a^2\geq0$ for $2\leq x<5/2$. Moreover, 
$$
b^2-a^2 = -x^2(4 y^2-5 x-2 x^2) (4 y^2+5 x-2 x^2)\big((75-16x^2) y^2+25x^2\big),
$$
where clearly
$$
	4y^2+5x-2x^2\ge4y^2-5x-2x^2\ge 0.
$$
Hence $\frac{\partial\psi}{\partial x}(x,y)\le 0$ is equivalent to $(75-16 x^2) y^2+25 x^2\ge 0$. Since $y\geq 2>5/4$, we deduce that
$$
\begin{aligned}
\frac{\partial}{\partial x}\psi(x,y) &\le 0\qquad\text{if}
	\quad 2\le x\le\sqrt{\frac{75y^2}{16y^2-25}}\, ,\\
 \frac{\partial}{\partial x}\psi(x,y) &\geq 0\qquad\text{if}
 	\quad \sqrt{\frac{75y^2}{16y^2-25}}\leq x \le\frac52\, ,
\end{aligned}
$$
so that $\psi(x,y)$ is minimized when $x= \sqrt{75}y/\sqrt{16y^2-25}$, whose RHS is always in $(2,5/2]$ for $y\geq 5/2$, i.e., $\rho\geq 2$.
\end{proof}

\begin{proposition}
The conjecture (C) holds for $1<\rho\leq 2.57$. 
\end{proposition}

\begin{proof}
In Prop.\ \ref{prop:diagnew}, we showed that (C) holds for all $\rho\in(1,\sqrt2]$ and for all $(\rho,r)$ with $1/\sqrt\rho<r\leq r_3(\rho)$. Moreover, taking $\tau=1/\sqrt2$ in Lemma \ref{lem:lastpatchmono}, we obtain for $r_3(\rho)\leq r\leq 1$ that
$$
	\|X A X^{-1}\|^2\le\max \left(\psi(2,y),\psi(5/2,y)\right),
$$
and we proceed to show that the maximum equals $\psi(2,y)$ for all $y=\rho+1/\rho<2.96$.

Defining $F$ by
$$
	F(y) := \psi(2,y) -\psi(5/2,y) = \frac{1}{200} \left(61 y^2 - 75 y \sqrt{y^2-4}-50\right),
$$
we have $F'(y) < 0$ and $F(2.96) >0$. Then $\psi(r,y) \leq\psi(2,y)$ for all $1/\sqrt 2\leq r\leq 1$ and $\rho\leq 2.57$ by Lemma \ref{lem:lastpatchmono}, and the result now follows like in Prop.\ \ref{prop:caser=1}.
\end{proof}

The next step is:

\begin{proposition}
The conjecture (C) holds for $5/2<\rho<10$. 
\end{proposition}

\begin{proof}
In the notation of Prop.\ \ref{prop:smallr}, $p(x,y)=0$ can equivalently be written 
\begin{equation}\label{eq:P31alt}
	(\rho+1/\rho)^2 = \frac{(1+r^4) (\rho^2+4-12 r^4)}{4r^4}\,,
\end{equation}
and thus $r_1(\rho)$ is the unique positive root $r$ of this equation. By Lemma \ref{lem:lastpatchmono}, it suffices to show that $\psi(r^2+1/r^2,\rho+1/\rho)\le\rho^2/4$ for $r=r_1(\rho)$. In Fig.\ \ref{fig:X2} below, we give numerical evidence for this inequality by plotting $4\psi/\rho^2$ on the curve $r=r_1(\rho)$ for $5/2\le\rho\le 10$.

\begin{figure}[!h]
\begin{center}
\includegraphics[width=0.7\textwidth]{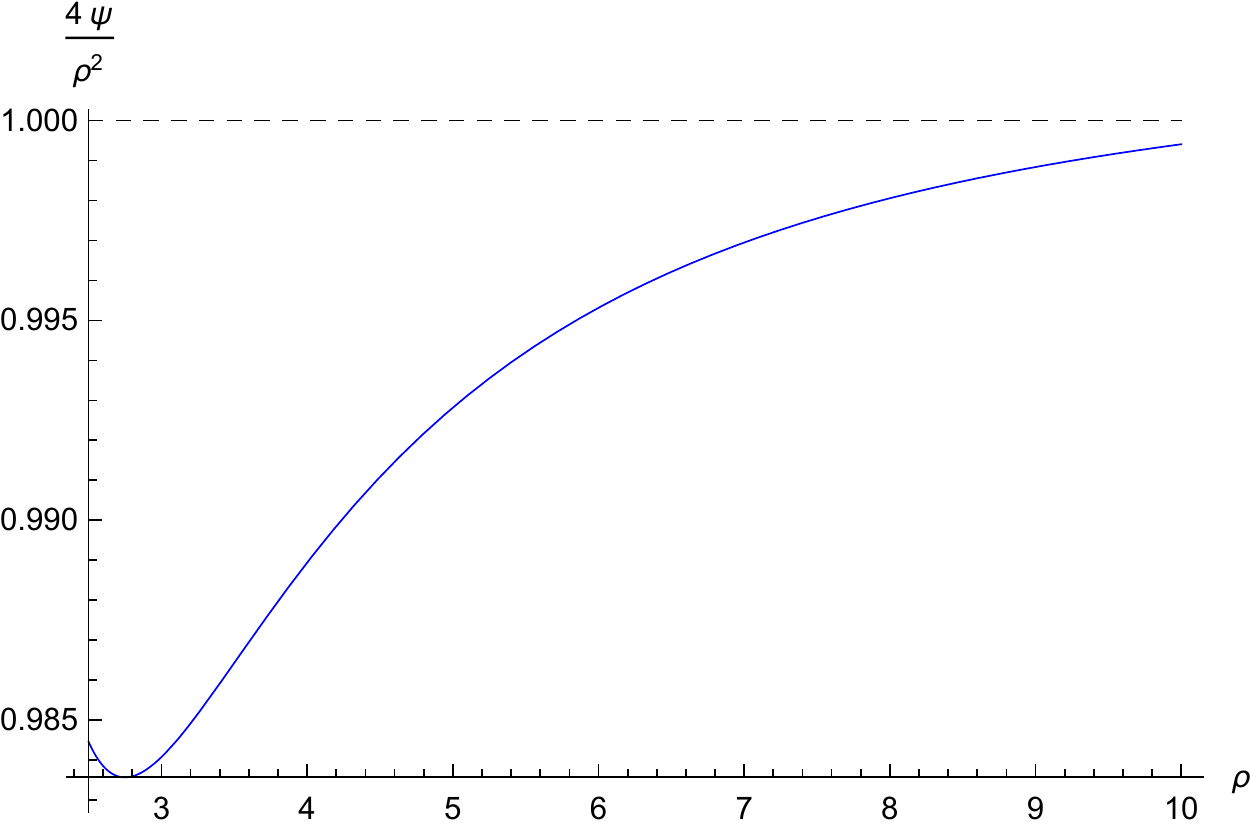}
\caption{The values of $4 \|X A X^{-1}\|^2/\rho^2$ on the curve $r=r_1(\rho)$ for $5/2\le\rho\le 10$ compared to $1$. }
\label{fig:X2}
\end{center}
\end{figure}

Now we proceed to give logical proof of $4\psi/\rho^2\leq 1$. By \eqref{eq:P31alt},
$$
y^2 = \frac{(1+r^4) (\rho^2+4-12 r^4)}{4r^4}\qquad\text{and}\qquad y^2-x^2 = \frac{(1+r^4) (\rho^2-16 r^4)}{4r^4} \, ,
$$
which satisfy $y^2>0$ because $r_1(\rho)<1/\sqrt[4]3$, $y^2-x^2>0$ since additionally $\rho\geq 5/2$, and
$$
\psi(x,y)=\frac{10r^2+5r^2\rho^2-70r^6-\sqrt{(\rho^2+4-12r^4)(\rho^2-16r^4)(-4+17r^4-4r^8)}}{4(1+r^4)^2}\, ,
$$
where one sees that $-4+17r^4-4r^8\geq0$ for $r^4\in[1/4,1/3]$ by calculating the values at the end points. We shall next show that 
$$
	10r^2+5r^2\rho^2-70r^6-(1+r^4)^2\rho^2\geq0
$$ 
for $\rho\in[5/2,10]$ and $r=r_1(\rho)$, and then the condition $\psi(x,y)\le\rho^2/4$ is equivalent to
\begin{equation}\label{eq:bigineq}
\begin{aligned}
&\left(10r^2+5r^2\rho^2-70r^6-(1+r^4)^2\rho^2\right)^2 \\
&\qquad\qquad\qquad
	-\left(\sqrt{(\rho^2+4-12r^4)(\rho^2-16r^4)(-4+17r^4-4r^8)}\right)^2\le 0\, .
\end{aligned}
\end{equation}

Solving \eqref{eq:P31alt} for $\rho^2>1$ gives that $r=r_1(\rho)$ implies
\begin{equation}\label{eq:rhosq}
	\rho^2 = 2\, \frac{3 r^8+4r^4-1+\sqrt{(1+r^4)(1-8r^4+15r^8+9r^{12})}}{1-3r^4}\, ,
\end{equation} 
and inserting \eqref{eq:rhosq} into the inequality $10 r^2 - 70 r^6 + 5 r^2 \rho^2 - (1 + r^4)^2 \rho^2\geq0$ gives
$$
\frac{p_1(r^2) + p_2(r^2) \sqrt{p_3(r^2)}}{1 - 3 r^4}\geq0\, ,
$$
where
$$
\begin{aligned}
p_1(t) &:=2 - 4 t^2 - 60 t^3 - 20 t^4 + 240 t^5 - 20 t^6 - 6 t^8\, ,\\
 p_2(t) &:= -2 +10 t -4 t^2 -2 t^4\,, \qquad\text{and}\\
 p_3(t) &:= (1+t^2)(1-8t^2+15t^4+9t^6)
\end{aligned}
$$
satisfy $p_1(1/2) + p_2(1/2) \sqrt{p_3(1/2)}=0$. Furthermore, the polynomials $p_1$, $p_2$, and $p_3$ are all increasing on $1/2\le t<1/\sqrt3$, because $p_3'(t)\geq p_3'(1/2)=25/16$, $p_2'(t)\geq p_2'(1/\sqrt3)>3$, and in a similar fashion, $p_1^{(k)}>0$ for $k=5,4,3,2,1$.

The inequality \eqref{eq:bigineq} can be written $(1+r^4)^2 B(r,\rho)\le 0$, where
$$
B(r,\rho):=\rho^4 (5-10r^2+2 r^4+r^8)+4 (5r^2-4)(7r^4-1)\rho^2 +12r^4(64r^4-13)\, .
$$
We are thus done once we have proved that $B(r,\rho)\leq 0$ for the relevant values of $r$ and $\rho$. Substituting \eqref{eq:rhosq} into $B(r,\rho)$ gives (with $1-3r^4>0$)
$$
\frac{(1-3r^4)^2}4\, B(r,\rho) = p_4(r^2)+p_5(r^2)\sqrt{p_3(r^2)}\, ,
$$
where
$$
\begin{aligned}
p_4(t) :=&\, 2-10t+2t^2+10t^3+19 t^4+410t^5-812t^6-1020t^7\\
&+2435t^8-810t^9+84t^{10}+18t^{12}\, \qquad\text{and}\\
p_5(t) :=& -2+10t-44t^2+20t^3+212t^4-270 t^5+20t^6+6t^8\, .
\end{aligned}
$$
Therefore, the proof is complete once we show that
\begin{equation}\label{eq:B}
p_4(t)+p_5(t)\sqrt{p_3(t)}\le 0,\qquad t\in \big[1/2,r_1(10)^2\big)\, ;
\end{equation}
recall that $r_1(\rho)$ is increasing in $\rho$. With $p$ in \eqref{eq:psmallr}, we have $p(2/\sqrt7,10)>0$ and by the proof of Prop.\ \ref{prop:smallr}, $p(r,10)$ is increasing in $r$, so that $r_1(\rho)^2\leq4/7$.

For $1/2\le t\le 0.5327$, we consider $p_6(t) := p_4(t) - 2 t^{12}$. It is straightforward to check that $p_6^{(n)}(t)> 0$ for $n=3,\ldots,10$, and $p_6^{(n)}(t) < 0$ for $n=1,2$. Hence $p_6(t)\geq p_6(0.5327) > 0$ for $1/2\le t\le 0.5327$. Defining $p_7(t) := p_4(t) -  t^{12}$, $0.5327\le t\le 4/7$, it can analogously be shown that $p_4(t) > p_7(t) > 0$ by verifying that $p_7^{(n)}(t) > 0$ for $n=2,\ldots,10$, and $p'_7(t) < 0$; thus $p_4(t)>0$ on the interval  $1/2\le t\le 4/7$. Furthermore, $p_5(t)$ is negative and increasing in $t\in[1/2,4/7]$; indeed, $p_5^{(n)} > 0$ for $n=1,6$ and $p_5^{(n)} < 0$ for $n=2, \ldots, 5$, and we conclude that $p_5(t)\leq p_5(4/7) < 0$ for $1/2\le t\le 4/7$. Consequently, \eqref{eq:B} is equivalent to
$$
p_8(t):=p_4(t)^2-p_5(t)^2 p_3(t)\le 0,\qquad 1/2\leq t\leq 4/7\, ,
$$
where $p_8$ can be factorized as $p_8(t) = t^2 (2t-1)^2(1-3 t^2)^2 \,p_9(t)$ with
$$
\begin{aligned}
p_9(t) :=\, &-140+120 t+496 t^2-3116t^3+4400t^4+10364t^5-38295t^6\\
                 &+12584t^7+77722t^8-69288t^9+9009t^{10}+1728t^{11}+1728t^{12}\, .
\end{aligned}
$$

We finish the proof by establishing that $p_9(t)\le p_9(4/7)$, which is negative, on $[1/2,4/7]$. Using repeated differentiation again, we obtain  $p_9^{(10)} > 0,$ and then we deduce $p_9^{(9)}> 0,\, p_9^{(n)} < 0$ for $n=8,7,6$, and $p_9^{(n)}>0$ for $n=5,4,...,1$.
\end{proof}

The following proposition completes the proof of Thm \ref{mainthm}:

\begin{proposition}
The conjecture (C) holds for $\rho\geq10$.
\end{proposition}

\begin{proof}
By Prop.\ \ref{prop:case4}, (C) holds for $\rho\geq 10$ and $r\leq 0.77$; therefore fix $r=0.77$. By \eqref{eq:psinorm}, $\|X A X^{-1}\|^2$ is the smaller zero of the parabola
$$
Q(\lambda) := 4\big(4 x^4\lambda^2-20 x (2 y^2-x^2) \lambda+25x^2+16y^4-16x^2y^2\big)\, .
$$
If the inequality $Q(\rho^2/4)\le 0$ holds for a given $\rho\geq10$, then $\|X A X^{-1}\|^2\le\rho^2/4$, and hence $c^2 \|X A X^{-1}\|^2\le 1$ is satisfied at $r=0.77$ by Lemma \ref{lem:geom}. In this case, Lemma \ref{lem:lastpatchmono} and Prop.\ \ref{prop:caser=1} establish (C) for all $0.77\leq r\leq 1$.

In order to verify that $Q(\rho^2/4)\le 0$ for $r=0.77$ and $\rho\geq10$, we introduce a positive, strictly increasing function 
$$
	\alpha(\rho ) := \frac{\rho^2}{y^2} = \frac{1}{(1+1/\rho^{2})^2}<1,
$$	
such that
$$
	Q\left(\frac{\rho^2}{4}\right) = Q\left(\frac{\alpha y^2}{4}\right) = 100 x^2-(b-a y^2) \,y^2\, ,
$$
where
$$
	a := x^4\alpha^2-40x\alpha+64,\qquad  \text{and}\qquad b:= 64 x^2-20x^3\alpha.
$$
From $\alpha\in(0,1)$ and $x=0.77^2+1/0.77^2$, we obtain
$$
	b\ge (64-20x)x^2 > 18x^2,
$$
and $a$ is a decreasing function of $\rho$ with negative value for $\rho=21$. Hence $a<0$ for $\rho\geq21$ and then (using $\rho\geq10$)
$$
	 Q\left(\frac{\rho^2}4\right)\le 100x^2-by^2\le (100-18y^2) x^2 < 0.
$$
This establishes (C) for $\rho\geq21$.
 
When $\rho=20$, $a>0$, so that $a>0$ whenever $\rho\in[10,20]$. For $10\le\rho_{-}\le\rho\le\rho_{+}$ and $\rho_{-}\le20$, it therefore holds that
 $$
 Q\left(\frac{\rho^2}4\right)\le H(\rho_{-},\rho_{+}) := 
 100x^2-b_-y_-^2+a_+y_{+}^4\, ,
 $$
where $y_{\pm} := \rho_{\pm}+1/\rho_{\pm}\, ,\ \alpha_{-} := \rho_{-}^2/y_{-}^2\, ,\ a_+ := x^4\alpha_{-}^2-40x\alpha_{-}+64\, ,$ and $b_-:=(64-20x)x^2$. Now we compute
$H(16,21) \approx -2524$, $H(14,16) \approx -5167$, $H(12,14) \approx -274$, $H(11,12) \approx -1994$, and $H(10,11) \approx -721$, from which it follows that $Q(\rho^2/4)\leq 0$ for $\rho$ in each interval $[16,21]$, $[14,16]$, $[12,14]$, $[11,12]$, and $[10,11]$.
\end{proof}

\appendix

\section{Proof of Observation \ref{obs:ChoiExt}}\label{sec:Choi}

We are required to prove that \eqref{eq:Cineq} holds for $a I+DP$ and $a I+PD$, where $a\in\C$, $D$ is diagonal, and $P$ is a permutation matrix. WLOG, $a=0$ and next we observe that \eqref{eq:Cineq} holds for $DP$ if and only if \eqref{eq:Cineq} holds for $PD$; indeed every permutation matrix $P$ is unitary and hence \eqref{eq:Cineq} holds for $DP$ if and only if \eqref{eq:Cineq} holds for $P(DP)P^*=PD$.

We now concentrate on the matrix $DP$. Let $U$ be a permutation matrix that collects the cycles in $P$, so that $UPU^*=\mathrm{blockdiag}\,(P_1,\ldots P_m)$, where each $P_k$ is a cyclic backwards shift, i.e., of the form
$$
	P_k = \bbm{0&1&0&\ldots&0 \\ 0&0&1&\ldots&0 \\ \vdots & \ddots & \ddots & \ldots & \vdots \\
		0&0&0&\ldots&1 \\ 1&0&0&\ldots&0}\qquad \text{or}\qquad P_k = I.
$$
Setting $\widetilde D:=UDU^*$ we obtain that $\widetilde D$ is also diagonal (with the diagonal some permutation of the diagonal of $D$), and it follows that there exist square matrices $A_1,\ldots,A_m$, such that
$$
	UDPU^* = \widetilde D \,\mathrm{blockdiag}\,(P_1,\ldots P_m) 
	= \mathrm{blockdiag}\,(A_1,\ldots A_m)=:A.
$$

Since all the blocks $A_k$ are of the form in \cite{Cho13}, it follows that $(C)$ holds for all $A_k$. Moreover, by \cite[Thm 5.1--2]{GR97}, it holds for all $k$ that $W(A_k)\subset W(A)$. Using again that every permutation matrix is unitary, we finally have that for every polynomial $p$:
\begin{equation}\label{eq:blockdiag}
\begin{aligned}
	\|p(DP)\| &= \|\mathrm{blockdiag}\,\big(p(A_1),\ldots, p(A_m)\big) \| \\
	&= \max_{k\in\set{1,\ldots,m}} \|p(A_k)\|  \\
	&\leq 2 \max_{k\in\set{1,\ldots,m}}\set{|p(z)|\bigmid z\in W(A_k)} \\
	&\leq 2 \max_{z\in W(A)}|p(z)| = 2 \max_{z\in W(DP)}|p(z)|.
\end{aligned}
\end{equation}

\section*{Acknowledgments}
The authors are most grateful to the referees for their generous help with improving the manuscript. In particular, one of the referees did the main part in finding the essential parametrization \eqref{eq:Adef}.

\def\cprime{$'$}
\providecommand{\bysame}{\leavevmode\hbox to3em{\hrulefill}\thinspace}
\providecommand{\MR}{\relax\ifhmode\unskip\space\fi MR }
\providecommand{\MRhref}[2]{%
  \href{http://www.ams.org/mathscinet-getitem?mr=#1}{#2}
}
\providecommand{\href}[2]{#2}

\end{document}